\documentclass{svproc}
\usepackage{url}

\usepackage{graphicx}
\usepackage{multicol}
\usepackage{footmisc}

\usepackage{pgfplots}
\usepackage{subfigure}
\pgfplotsset{compat=1.12}

\usepackage{amsfonts,amsmath,amssymb}
\usepackage[latin1]{inputenc}
\usepackage{dsfont}
\usepackage{enumerate}
\usepackage{xcolor}
\usepackage[all]{xy}
\def\notshow#1\notshowend{} %
\usepackage{hyphenat}

\usepackage{graphicx}

\usepackage{tikz}
\usepackage{tikz-3dplot}
\usepackage{pgfplots}

\usepackage{multirow}
\usepackage{array}

\newcommand{\C}{\mathcal{C}}

\newcommand{\df}{\mathrm{d}}
\usepackage{amssymb}
\usepackage{amsfonts}
\usepackage{mathrsfs}
\usepackage{hyperref}
\usepackage{mathtools}

\usepackage[normalem]{ulem} 

\def\bb#1\eb{\textcolor{blue}{#1}} 
\def\br#1\er{\textcolor{red}{#1}} %
\def\bm#1\em{\textcolor{purple}{#1}} %

\newcommand{\R}{\mathds R}
\newcommand{\N}{\mathds N}

\newtheorem{thm}{Theorem}
\newtheorem{prop}{Proposition}
\newtheorem{lem}{Lemma}
\newtheorem{cor}{Corollary}
\newtheorem{defi}{Definition}

\newtheorem{rem}{Remark}

\newcommand{\ben}{\begin{enumerate}}
\newcommand{\een}{\end{enumerate}}
\newcommand{\bit}{\begin{itemize}}
\newcommand{\eit}{\end{itemize}}
\newcommand{\edoc}{\end{document}}

\begin{document}
\mainmatter              
\title{Lightlike hypersurfaces and time-minimizing geodesics in cone structures\thanks{This is a preprint of the following chapter: M. \'{A}. Javaloyes and E. Pend\'{a}s-Recondo, Lightlike Hypersurfaces and Time-Minimizing Geodesics in Cone Structures, published in Developments in Lorentzian Geometry, Springer Proceedings in Mathematics \& Statistics, vol. 389, edited by A. L. Albujer et al., 2022, Springer, reproduced with permission of Springer Nature Switzerland AG. The final authenticated version is available online at: http://dx.doi.org/10.1007/978-3-031-05379-5.}}
\titlerunning{Lightlike hypersurfaces and time-minimizing geodesics}  
%
\author{Miguel \'{A}ngel Javaloyes \and Enrique Pend\'{a}s-Recondo}
\authorrunning{M. \'{A}. Javaloyes and E. Pend\'{a}s-Recondo} 
%
%
\institute{Departamento de Matem\'{a}ticas, Universidad de Murcia, Campus de Espinardo, 30100 Espinardo, Murcia, Spain\\
\email{majava@um.es}\\
\email{e.pendasrecondo@um.es}}

\maketitle              

\begin{abstract}
Some well-known Lorentzian concepts are transferred into the more general setting of cone structures, which provide both the causality of the spacetime and the notion of cone geodesics without making use of any metric. Lightlike hypersurfaces are defined within this framework, showing that they admit a unique foliation by cone geodesics. This property becomes crucial after proving that, in globally hyperbolic spacetimes, achronal boundaries are lightlike hypersurfaces under some restrictions, allowing one to easily obtain some time-minimization properties of cone geodesics among causal curves departing from a hypersurface of the spacetime.
\end{abstract}

\section{Introduction}
One of the key properties of Finsler metrics is that their indicatrix is transversal to the position vector. Indeed, this property allows for a more general definition of Finsler Geometry as in \cite{Bryant02} based on the indicatrix. But when one considers pseudo-Finsler metrics and allows for null directions, this property is lost. Indeed, the null cone (the subset of lightlike directions) is always tangent to the position vector and, in fact, it always contains the semi-line from the origin. Unlike the first case in which the indicatrix is transversal to the position vector, the null cone does not determine the pseudo-Finsler metric in some open subset of the tangent bundle, and it turns out that the lightlike geodesics are determined only up to reparametrization. Nevertheless, using a certain quotient space, it is possible to define some kind of curvature invariants (see \cite{Mak18}), and on the other hand, the focal points of these lightlike geodesics do not depend on the pseudo-Finsler metric used to compute them \cite{JavSoa20}. Additionally, if these null cones enclose a convex subset in every tangent space, then it is possible to study their causal relations, namely, the connections between points by means of curves whose tangent vectors lie always inside the null cones. A study of causal properties from this general point of view was undertaken in \cite{FS}, and then applied to Finsler spacetimes in \cite{JS14}. Since then, there has been  a renewed interest in the so-called cone structures \cite{BS18,BS20,Min,JS20}.   Remarkably, these cone structures can be used to solve Zermelo's navigation problem \cite{JS20} and to describe a time-dependent wavefront, e.g., sound waves or wildfires in the presence of wind (see \cite{JPS} and references therein). Our aim in this work is to adapt and generalize some specific Lorentzian notions and results to the cone structures framework.

The paper is structured as follows. We start in \S \ref{sec:preliminary} introducing the notion of cone structures (Def. \ref{def:cone}), following mainly \cite{JS20}. This allows one, without the need of a metric, to establish a causality on the spacetime, making a distinction between timelike, lightlike and spacelike vectors (Def. \ref{def:causality}), and even to define a generalized notion of geodesic: the cone geodesics (Def. \ref{def:cone_geod}). Anyway, working with a specific metric will enable us to obtain results that, although more general, resemble the Lorentzian ones. To this end, we remark the relationship between both notions: a Lorentz-Finsler metric uniquely determines a cone structure, and a cone structure uniquely defines a class of Lorentz-Finsler metrics whose lightlike pregeodesics coincide with the cone geodesics of the cone structure (Thm. \ref{th:cone_geod}).

Then, in \S \ref{sec:light_hyp} we move on to define lightlike hypersurfaces within this framework (Def. \ref{def:lig_hyp}). In particular, we show that our definition can be expressed in terms of a Lorentz-Finsler metric, therefore generalizing the Lorentzian concept (Prop. \ref{prop:lig_hyp}), and that some of the usual Lorentzian properties of this type of hypersurfaces are still valid here (Prop. \ref{prop:int_cur_geod} and \ref{prop:tra_int}).

Next, in \S \ref{sec:smooth} we focus on the smoothness of achronal boundaries, which are, in general, (non-smooth) topological hypersurfaces of the spacetime. In order to ensure the smoothness we need the spacetime $ M $ to be globally hyperbolic, which implies that $ M $ is topologically a product $ \R \times D $, with the projection $ t: M \rightarrow \R $ being a temporal function. Then, for any compact hypersurface $ S $ of $ M $ included in the slice $ \{t=0\}:=\{0\}\times D $, $ \partial I^+(S)  \setminus S  $ is a lightlike hypersurface for small $ t $ (Thm. \ref{th:lig_hyp}).

Finally, this allows us in \S \ref{sec:min_geod} to immediately prove some results regarding the time-minimization properties of cone geodesics. Specifically, we show that any causal curve entirely contained in $ \partial I^+(S) $ must be a cone geodesic (Prop. \ref{prop:cau_cur}) that arrives earlier at any of its points than any other causal curve departing from $ S $ (Thm. \ref{th:min_time}).  Even though we provide self-contained proofs based on the smoothness of $ \partial I^+(S) \setminus S $, the relationship of these results with the ones in \cite{JPS} is discussed in Rem. \ref{rem:rel_huygens}. 

Fig. \ref{fig} summarizes the main results and depicts the basic geometrical picture one should have in mind throughout this work.

\begin{figure}
\centering
\includegraphics[width=0.7\textwidth]{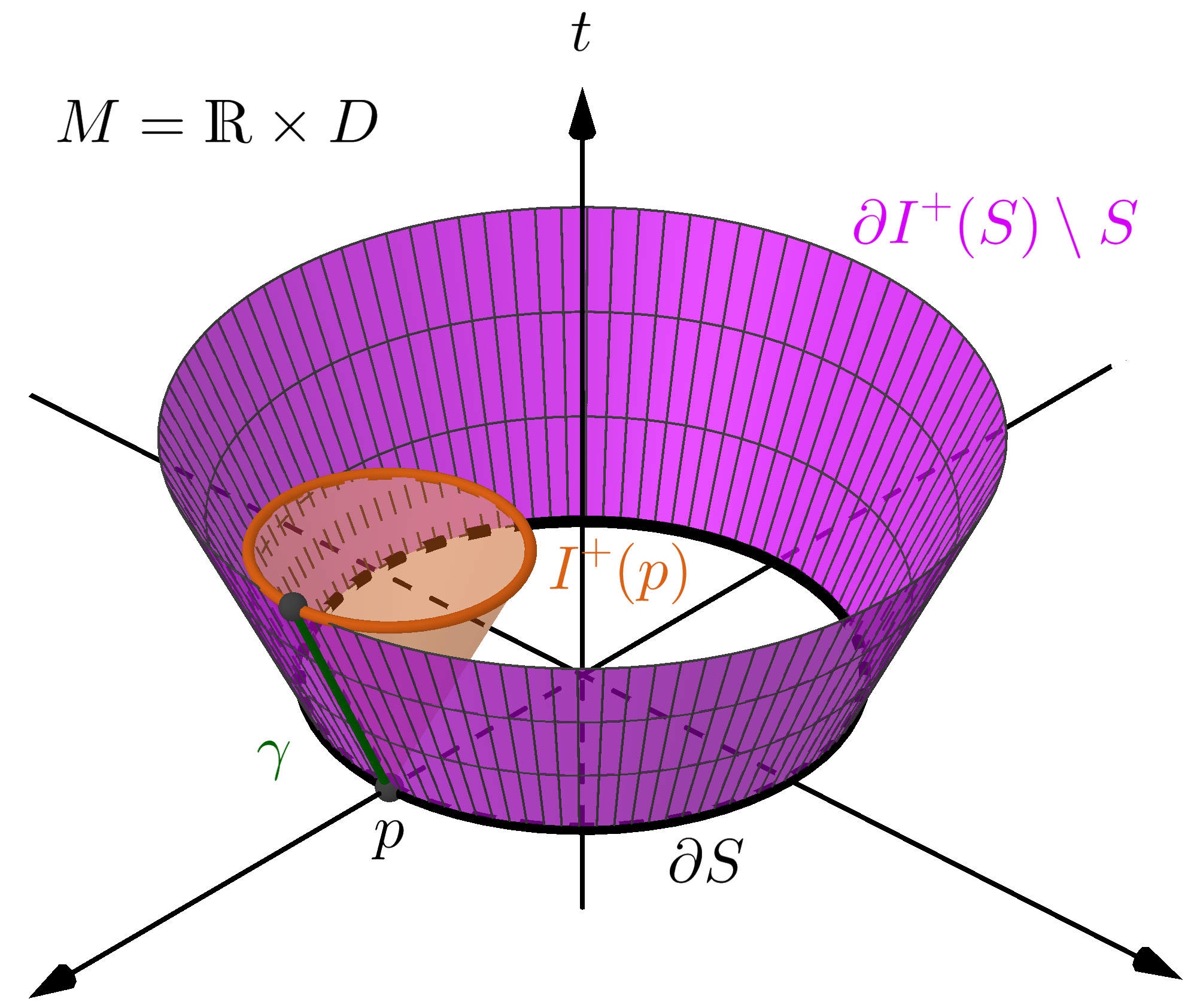}
\caption{The cone structure provides the causality of the spacetime $ M $, including the chronological futures $ I^+(p) $. When $ M $ is globally hyperbolic, it can be decomposed as a product $ \R \times D $. In this case, $ \partial I^+(S)  \setminus S $ becomes a lightlike hypersurface, at least for small $ t $. Moreover, any causal curve $ \gamma $ entirely contained in $ \partial I^+(S) $ is a cone geodesic that minimizes the propagation time from $ S $.}
\label{fig}
\end{figure}

\section{Preliminary notions on cone structures}
\label{sec:preliminary}
This section summarizes the main definitions regarding cone structures and its relation to Lorentz-Finsler metrics, following \cite{JS20}. Let $ V $ and $ M $ be, respectively, a real vector space and a smooth (namely, $C^\infty$) manifold of dimension $ m = n+1 \geq 3 $, being $ TM $ the tangent bundle of $ M $.

\begin{defi}
\label{def:cone}
\begin{enumerate}[(i)]
\item A hypersurface\footnote{Throughout this work, every hypersurface or submanifold will be assumed smooth and embedded, unless otherwise specified.} $ \mathcal{C}_0 $ of $ V \setminus \lbrace 0 \rbrace $ is a {\em cone} if it satisfies the following properties:
\begin{enumerate}[(a)]
\item {\em Conic}: for all $ v \in \mathcal{C}_0 $, $ \lbrace \lambda v: \lambda > 0 \rbrace \subset \mathcal{C}_0 $.
\item {\em Salient}: if $ v \in \mathcal{C}_0 $, then $ -v \notin \mathcal{C}_0 $.
\item {\em Convex interior}: $ \mathcal{C}_0 $ is the boundary in $ V \setminus \lbrace 0 \rbrace $ of an open subset $ A_0 \subset V \setminus \lbrace 0 \rbrace $ (the $ \mathcal{C}_0 $-interior) which is convex, i.e., for any $ v,u \in A_0 $, the segment $ \lbrace \lambda v + (1-\lambda)u: 0 \leq \lambda \leq 1 \rbrace \subset V $ is included entirely in $ A_0 $.
\item {\em (Non-radial) strong convexity}: the second fundamental form of $ \mathcal{C}_0 $ as an affine hypersurface of $ V $ is positive semi-definite (with respect to an inner direction pointing out to $ A_0 $) and its radical at each $ v \in \mathcal{C}_0 $ is spanned by the radial direction $ \lbrace \lambda v: \lambda > 0 \rbrace $.
\end{enumerate}

\item A hypersurface $ \mathcal{C} $ of $ TM \setminus \textup{\textbf{0}} $ is a {\em cone structure} if for each $ p \in M $:
\begin{enumerate}[(a)]
\item $ \mathcal{C} $ is transverse to the fibers of the tangent bundle, i.e., if $ v \in \mathcal{C}_p \coloneqq T_pM \cap \mathcal{C} $, then $ T_v(T_pM) + T_{(p,v)}\mathcal{C} = T_{(p,v)}(TM) $,\footnote{This condition is necessary to ensure that the fibers $\C_p$ vary smoothly with $p\in M$  (see \cite[Rem. 2.8]{JS20}).} and
\item $ \mathcal{C}_p $ is a cone in $ T_pM $.
\end{enumerate}
We denote by $ A_p $ the $ \mathcal{C}_p $-interior, and $ A \coloneqq \cup_{p \in M} A_p $ is the {\em cone domain}.
\end{enumerate}
\end{defi}

More intuitively (see Fig. \ref{Fig1}), any cone can be constructed by taking a compact strongly convex hypersurface $ \Sigma_0 $ of an affine hyperplane $ \Pi \subset V $, with $ 0 \notin \Pi $, and taking all the open half-lines through $ \Sigma_0 $ starting at $ 0 $ \cite[Lem. 2.5]{JS20}.
	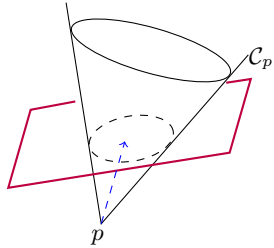
\begin{figure}
	\centering
\begin{tikzpicture}[scale=0.6,rotate around={-16:(-3.88,2)}]
\draw
(-6.1,4.5) -- (-4,0) -- (-1.9,4.5);
\draw[thick,color=purple] (-5.3,2.45) -- (-6.2,2) -- (-6.2,0.2) -- (-1.7,2.3) -- (-1.7,4) -- (-2.2,3.77);
\draw[blue,dashed,->] (-4,0) -- (-4,1.9); 
\draw (-4,4) ellipse (1.83cm and 0.5cm);
\draw[dashed,rotate around={26:(-3.88,2)}] (-3.88,2) ellipse (0.943cm and 0.5cm);
\node at (-4,-0.3)(s){$p$};
\node at (-1.6,4.5)(s){$\C_p$};


\end{tikzpicture}
\caption{Strongly convex cone structure cut by an affine hyperplane with no intersection (as a vector space) with the cone.}
\label{Fig1}
\end{figure}

Cone structures provide some classes of privileged vectors, which can be used to define notions that generalize those in the causal theory of classical spacetimes.

\begin{defi}
\label{def:causality}
Given a cone structure $ \mathcal{C} $ in $ M $, we say that a vector $ v \in T_pM $ is
\begin{itemize}
\item {\em timelike} if $ v \in A_p $,
\item {\em lightlike} if $ v \in \mathcal{C}_p $,
\item {\em causal} if it is timelike or lightlike, i.e., if $ v \in \overline{A}_p \setminus \lbrace 0 \rbrace $,\footnote{It is usual to distinguish {\em future-directed} vectors (when $ v \in \overline{A}_p \setminus \lbrace 0 \rbrace $) from {\em past-directed} ones (when $ -v \in \overline{A}_p \setminus \lbrace 0 \rbrace $). Nevertheless, this distinction will not be necessary here, as we will only focus on future-directed vectors and, according to our definition, a causal vector is always future-directed.}
\item {\em spacelike} if neither $ v $ nor $ -v $ is causal.
\end{itemize}

Analogously, we say that a piecewise smooth curve $ \gamma: I \rightarrow M $ is timelike, lightlike, causal or spacelike, when its tangent vector $ \gamma' $ (or both $ \gamma'(t_0^+) $ and $ \gamma'(t_0^-) $ at any break $ t_0 \in I $) is timelike, lightlike, causal or spacelike, respectively. This allows us to define the following sets:
\begin{itemize}
\item {\em chronological future}: $ I^+(p) \coloneqq \lbrace q \in M: \exists \text{ timelike curve from } p \text{ to } q \rbrace $,
\item {\em chronological past}: $ I^-(p) \coloneqq \lbrace q \in M: \exists \text{ timelike curve from } q \text{ to } p \rbrace $,
\item {\em causal future}: $ J^+(p) \coloneqq \lbrace q \in M: p=q \text{ or } \exists \text{ causal curve from } p \text{ to } q \rbrace $,
\item {\em causal past}: $ J^-(p) \coloneqq \lbrace q \in M: p=q \text{ or } \exists \text{ causal curve from } q \text{ to } p \rbrace $.
\end{itemize}
Also, we say that two points $ p,q $ are horismotically related, denoted $ p \rightarrow q $, if $ q \in J^+(p) \setminus I^+(p) $.

Finally, a {\em temporal function} is a smooth real function $ t: M \rightarrow \R $ such that it is strictly increasing when composed with timelike curves and no causal vector is tangent to the slices $ \lbrace t = \text{constant} \rbrace $.
\end{defi}

Bear in mind that the chronological sets $ I^{\pm}(p) $ are always open, in fact $ I^{\pm}(S) = \textup{Int}(J^\pm(S)) $ for any $ S \subset M $,  and that  the topological closure and boundary of the chronological and causal sets coincide, i.e., $ \overline{I^\pm(S)}=\overline{J^\pm(S)} $ and $ \partial I^\pm(S)=\partial J^\pm(S) $  (see \cite[Corollary 6.6]{AJ} and \cite[Thm. 2.7]{Min}). The latter sets are called {\em achronal boundaries} and they indeed have the property of being achronal, i.e., no two points in them can be joined by a timelike curve (see   \cite[Thm. 6.9]{AJ} and  \cite[Prop. 2.13]{Min}).

Cone structures also admit the following notion of geodesic.

\begin{defi}
\label{def:cone_geod}
Let $ \mathcal{C} $ be a cone structure in $ M $. A continuous curve $ \gamma: I \rightarrow M $ is a {\em cone geodesic} if it is locally horismotic, i.e., for each $ t_0 \in I $ there exists an open neighborhood $ U $ of $ \gamma(t_0) $ such that, if $ I_{\varepsilon} \coloneqq [t_0-\varepsilon,t_0+\varepsilon] \cap I $ satisfies $ \gamma(I_{\varepsilon}) \subset U $ for some $ \varepsilon > 0 $, then
$$
t_1 < t_2 \Leftrightarrow \gamma(t_1) \rightarrow_U \gamma(t_2), \quad \forall t_1,t_2 \in I_{\varepsilon},
$$
where $ \rightarrow_U $ is the horismotic relation for the natural restriction $ \mathcal{C}_U $ of the cone structure to $ U $.
\end{defi}

Now let us define Lorentz-Finsler metrics, which are strongly related to cone structures, and provide a link between both notions.

\begin{defi}
\begin{enumerate}[(i)]
\item A positive function $ L: A_0 \subset V \setminus \lbrace 0 \rbrace \rightarrow \R^+ $ is a {\em Lorentz-Minkowski norm} if
\begin{enumerate}
\item $ A_0 $ is a conic domain, i.e., $ A_0 $ is open, non-empty, connected and if $v\in A_0$, then $\lambda v\in A_0, \forall \lambda>0$,
\item $ L $ is smooth and positively two-homogeneous, i.e., $ L(\lambda v) = \lambda^2L(v) $ for all $ v \in A_0, \lambda > 0 $,
\item for every $ v \in A_0 $, the {\em fundamental tensor} $ g_v $, given by
\begin{equation*}
g_v(u,w) = \left.\frac{1}{2}\frac{\partial^2}{\partial\delta\partial\eta}L(v + \delta u + \eta w)\right\rvert_{\delta=\eta=0}, \quad \forall u,w \in V,
\end{equation*}
has index $ n $, and
\item the topological boundary $ \mathcal{C}_0 $ of $ A_0 $ in $ V \setminus \lbrace 0 \rbrace $ is smooth and $ L $ can be smoothly extended as zero to $ \mathcal{C}_0 $ with non-degenerate fundamental tensor.
\end{enumerate}

\item A positive function $ L: A \subset TM \setminus \textup{\textbf{0}} \rightarrow \R^+ $ is a {\em Lorentz-Finsler metric} if
\begin{enumerate}[(a)]
\item $ \overline{A} \setminus \textup{\textbf{0}} $ is a submanifold of $ TM \setminus \textup{\textbf{0}} $ with boundary $ \mathcal{C} $,
\item each $ L_p \coloneqq L\rvert_{A_p} $ is a Lorentz-Minkowski norm for all $ p \in M $, and
\item $ L $ is smooth and can be smoothly extended as zero to $ \mathcal{C} $.
\end{enumerate}
In this case, we say that $ (M,L) $ is a {\em Finsler spacetime}.
\end{enumerate}
\end{defi}

\begin{rem}
Since any Lorentz-Finsler metric $ L: A \rightarrow \R^+ $ is smooth on $ \mathcal{C} $ with non-degenerate fundamental tensor, $ L $ can always be smoothly extended to an open conic domain $ A^* $ containing $ \overline{A} \setminus \{0\} $ in such a way that $ L $ has index $ n $ on $ A^* $ and $ L < 0 $ outside $ \overline{A} \cup -\overline{A} $.\footnote{Anyway, the smoothness on all $ A^* $ will not be strictly needed here, just on a neighbourhood of $ \mathcal{C} $, as we will only focus on lightlike directions.}
\end{rem}

Each Lorentz-Finsler metric $ L: A \rightarrow \R^+ $ determines a unique cone structure. Indeed, the boundary $ \mathcal{C} $ of $ A $ in $ TM \setminus \textbf{0} $ is a cone structure with cone domain $ A $ \cite[Cor. 3.7]{JS20}. Conversely, each cone structure $ \mathcal{C} $ uniquely determines a (non-empty) class of anisotropically equivalent Lorentz-Finsler metrics \cite[Rem. 5.9]{JS20}, any of which will be called {\em compatible with} $ \mathcal{C} $. All these metrics share the same lightlike pregeodesics (see, e.g., \cite[Prop. 3.4]{JavSoa20}),\footnote{The geodesics of a Lorentz-Finsler metric can be defined as the critical curves of its energy functional. Moreover, they determine a spray and can also be obtained as the autoparallel curves of several classical connections (Berwald, Cartan,  Chern and Hashiguchi), having all of them the same associated non-linear connection.} which coincide with the cone geodesics of $\C$, as shown by the following result \cite[Thm. 6.6]{JS20}.

\begin{thm}
\label{th:cone_geod}
A curve $ \gamma: I \rightarrow M $ is a cone geodesic of $ \mathcal{C} $ if and only if $ \gamma $ is a lightlike pregeodesic for one (and then, for all) Lorentz-Finsler metric $ L $ compatible with $ \mathcal{C} $.
\end{thm}


\section{Lightlike hypersurfaces}
\label{sec:light_hyp}
Our goal in this section is to transfer the notion and properties of lightlike hypersurfaces, very well known in Lorentzian geometry (see \cite{G,O}), to the more general setting of cone structures. In what follows, $ M $ will denote a smooth manifold of arbitrary dimension $ m=n+1 \geq 3 $ endowed with a cone structure $ \mathcal{C} $, and $ L: A^* \rightarrow \R $ will be any Lorentz-Finsler metric compatible with $ \mathcal{C} $.

\begin{defi}
\label{def:lig_hyp}
A hypersurface $ H $ of $ M $ is {\em lightlike} if at each $ p \in H $ there exists a lightlike vector $ v \in T_pH $ such that $ T_v\mathcal{C}_p = T_pH $. In this case, we call the direction defined by $ v $ a {\em degenerate direction} of $ H $ at $ p $ (this direction will be proven unique below).
\end{defi}

Although the previous definition is expressed in terms of the cone structure $ \mathcal{C} $, the following result gives a characterization that uses the Lorentz-Finsler metric $ L $ and resembles the Lorentzian definition  (see, e.g., \cite[Def. 6.1]{G}). 
\begin{prop}
\label{prop:lig_hyp}
A hypersurface $ H $ is lightlike if and only if at each $ p \in H $ there exists $ v \in T_pH $ such that $ v $ is $ L $-orthogonal to $ u $, denoted $ v \bot_L u $, for all $ u \in T_pH $, i.e.,
\begin{equation*}
g_v(v,u) = \left. \frac{1}{2}\frac{\partial}{\partial \delta} L(v+\delta u)\right\rvert_{\delta = 0} = 0, \quad \forall u \in T_pH.
\end{equation*}
Moreover, the direction of $ v $ is a degenerate direction of $ H $ at $ p $ and the following properties hold:
\begin{enumerate}[(i)]
\item $ g_v|_{T_pH} $ is negative semi-definite, being the direction of $ v $ the only degenerate direction.

\item $ [v]^{\bot_L} \coloneqq \lbrace u \in T_pM: v \bot_L u \rbrace = T_pH $.

\item Every (nonzero) vector non-proportional to $ v $ is spacelike.
\end{enumerate}
\end{prop}
\begin{proof}
Let $ H $ be a lightlike hypersurface with a degenerate direction $ v \in T_pH $, so that $ T_pH = T_v\mathcal{C}_p $. Since $ v $ is lightlike, $ [v]^{\bot_L} = T_v\mathcal{C}_p $ (see \cite[Prop. 3.4 (iii)]{JS20}), so $ g_v(v,u) = 0 $ for all $ u \in T_v\mathcal{C}_p = T_pH $. Conversely, if there exists  $ v \in T_p H $ such that $ v \bot_L u $ for all $ u \in T_pH $, this means that $ T_pH \subset [v]^{\bot_L} $ and also $ L(v) = g_v(v,v) = 0 $, so $ v $ must be lightlike and thus $ [v]^{\bot_L} = T_v\mathcal{C}_p $. Since $ \text{dim}(T_pH)=\text{dim}(T_v\mathcal{C}_p) $, we conclude that $ [v]^{\bot_L} = T_pH = T_v\mathcal{C}_p $, i.e., $ H $ is a lightlike hypersurface, being the direction of $ v $ a degenerate direction. Finally, the claims (i), (ii) and (iii) are straightforward from \cite[Prop. 3.4]{JS20}.


\qed
\end{proof}

 The next results state some properties of lightlike hypersurfaces that will prove useful later on. Here we follow the Lorentzian results in \cite[\S 6]{G}, providing self-contained proofs to adapt them to our setting. 

\begin{cor}
\label{cor:N_unique}
Every lightlike hypersurface $ H $ determines a smooth lightlike vector field $ N $ on $ H $, unique up to a positive pointwise scale factor. Such an $ N $ will be said {\em associated with the lightlike hypersurface}.
\end{cor}
\begin{proof}
At each point $ p \in H $ there is a unique degenerate direction (Prop. \ref{prop:lig_hyp}(i)), which yields a unique lightlike vector $ v \in T_pH $ when normalized to $ \Omega(v)=1 $, being $ \Omega $ any timelike one-form, i.e., $ \Omega(u) > 0 $ for any causal vector $ u $ (such an $ \Omega $ can always be found; see \cite[Lem. 2.15]{JS20}). Therefore, $ N $ can be constructed by taking $ N_p \coloneqq v $ at each point $ p \in H $.

It remains to show that $ N $ is smooth. Obviously it suffices to prove it locally, since this is a local property. Let $ U \subset H $ be any open subset. Reducing $ U $ if necessary, we can assume that there exists a (smooth) spacelike frame $ \lbrace X^1,\ldots,X^{n-1} \rbrace $ on $ U $ such that at each $ p \in U $, $ \lbrace X^1_p,\ldots,X^{n-1}_p,N_p \rbrace $ is a basis of $ T_pH $. Define the map
\begin{equation*}
\begin{array}{cccc}
f \colon & TU \cap A^* \subset TH & \longrightarrow & \R^n\\
& (p,u) & \longmapsto & f(p,u) \coloneqq (g_u(u,X^1_p),\ldots,g_u(u,X^{n-1}_p),\Omega(u))
\end{array}
\end{equation*}
and note that $ f^{-1}(0,\ldots,0,1) = N|_U $. Its differential $ \df f_{(p,u)}: T_{(p,u)}TH \equiv T_pH \oplus T_u(T_pH) \rightarrow \R^n $ is a linear map satisfying
\begin{equation*}
\begin{split}
\df f_{(p,u)}(0,w) = & (\df f_p)_u(w) = \frac{\partial}{\partial\delta}f_p(u+\delta w)|_{\delta=0} = \\
= & (g_u(w,X^1_p),\ldots,g_u(w,X^{n-1}_p),\Omega(w))
\end{split}
\end{equation*}
for every $ p \in U $, $ u \in T_pH \cap A^*_p $, $ w \in T_u(T_pH) \equiv T_pH $. In particular, $ g_{N_p} $ is a (negative definite) scalar product on the subspace generated by $ \lbrace X_p^{i} \rbrace_{i=1}^{n-1} $ and thus $ \df f_{N_p}(0,X_p^1),\ldots,\df f_{N_p}(0,X_p^{n-1}),\df f_{N_p}(0,N_p) $ are linearly independent vectors. Indeed, if there exist $ \lambda_1,\ldots,\lambda_n \in \R $ such that
\begin{equation*}
\left\lbrace{
\begin{array}{l}
\sum_{i=1}^{n-1} \lambda_i g_{N_p}(X_p^i,X_p^j) + \lambda_n g_{N_p}(N_p,X_p^j) = 0, \quad \forall j=1,\ldots,n-1, \\
\sum_{i=1}^{n-1} \lambda_i \Omega(X_p^i) + \lambda_n \Omega(N_p) = 0,
\end{array}
}\right.
\end{equation*}
then from the first equation we obtain $ g_{N_p}(\sum_{i=1}^{n-1}\lambda_i X_p^i,X_p^j) = 0 $ for all $ j=1,\ldots,n-1 $, which means that $ \lambda_1 = \ldots = \lambda_{n-1} = 0 $ as $ g_{N_p} $ is non-degenerate on the subspace generated by $ \lbrace X_p^{i} \rbrace_{i=1}^{n-1} $. Also, since $ \Omega(N_p) = 1 $, from the second equation we obtain $ \lambda_n = 0 $. Therefore, we conclude that $ \df f_{N_p} $ is surjective for every $ p \in U $, i.e., $ (0,\ldots,0,1) $ is a regular value of $ f $ and thus $ N|_U $ is smooth.
\qed
\end{proof}


\begin{prop}
\label{prop:int_cur_geod}
Let $ H $ be a lightlike hypersurface and $ N $ its associated vector field. Then the integral curves of $ N $ are cone geodesics.
\end{prop}
\begin{proof}
Since cone geodesics are lightlike pregeodesics of $ L $ (recall Thm. \ref{th:cone_geod}), it suffices to show that $ \nabla_{N}^{N} N = \lambda N $ for some pointwise scale factor (i.e., real function) $ \lambda $, being $ \nabla $ the Chern connection of $ (M,L) $ (considered as a family of affine connections, as in \cite{J13}). Fix $ p \in H $ and choose any $ X_p \in T_pH $, so that $ g_{N_p}(N_p,X_p) = 0 $. We can extend $ X_p $ to a vector field $ X $ by making it invariant under the flow of $ N $, i.e.,
\begin{equation*}
[N,X] \coloneqq \nabla_{N}^{N} X - \nabla_{X}^{N} N = 0.
\end{equation*}
$ X $ remains tangent to $ H $, so along the flow line through $ p $, $ g_N(N,X) = 0 $. Differentiating we obtain
\begin{equation*}
0 = N(g_N(N,X)) = g_N(\nabla_{N}^{N} N,X) + g_N(N,\nabla_{N}^{N} X).
\end{equation*}
Rearranging and noting that $ \nabla_{N}^{N} X = \nabla_{X}^{N} N $ and $ g_N(N,N) = 0 $:
\begin{equation*}
g_N(\nabla_{N}^{N} N,X) = - g_N(N,\nabla_{X}^{N} N) = -\frac{1}{2}X(g_N(N,N)) = 0.
\end{equation*}
Hence $ g_{N_p}(\nabla_{N_p}^{N} N,X_p) = 0 $ for all $ X_p \in T_pH $, but since the direction of $ N_p $ is the only degenerate direction of $ H $ at $ p $, we conclude that $ \nabla_{N}^{N} N $ must be proportional to $ N $ at each point.
\qed
\end{proof}


\begin{prop}
\label{prop:tra_int}
Let $ \Sigma $ be the intersection of a lightlike hypersurface $ H $ (with associated vector field $ N $) with a hypersurface $ D $ of $ M $ which is transverse to $ N $ at every point $ p \in \Sigma $. Then $ \Sigma $ is a co-dimension two spacelike\footnote{Namely, every tangent vector is spacelike.} submanifold of $ M $, along which $ N $ is $ L $-orthogonal.
\end{prop}
\begin{proof}
$ D $ is transverse to $ N \in \mathfrak{X}(H) $ ($ N_p \notin T_pD $ for all $ p \in \Sigma $), so $ D $ is also transverse to $ H $ and by the transversality theorem, $ T_pD + T_pH = T_pM $ for all $ p \in \Sigma $. This ensures that $ \Sigma = H \cap D $ is a co-dimension two submanifold of $ M $ transverse to $ N $ and thus, by Prop. \ref{prop:lig_hyp}, spacelike and with $ N $ being $ L $-orthogonal to it (recall the definition of $ N $ in Cor. \ref{cor:N_unique}).
\qed
\end{proof}

\section{Smoothness of achronal boundaries}
\label{sec:smooth}
It is a standard result that, in general, achronal boundaries are (non-smooth) topological hypersurfaces (see \cite[Cor. 14.27]{O} for the Lorentzian result and \cite[Thm. 2.19]{Min} for its translation to cone structures). Here we will show that the smoothness can be guaranteed under certain conditions.

One of such conditions is the global hyperbolicity of the spacetime. This causality condition is very well known in Lorentzian geometry (see \cite{BEE,O} for background) and can be extended naturally to cone structures \cite{FS,JS14,Min}. So, from now on we will assume that $ M $ (endowed with the cone structure $ \mathcal{C} $) is globally hyperbolic. In particular, $ M $ admits a {\em Cauchy temporal function}, i.e., a surjective temporal function $ t: M \rightarrow \R $ (recall Def. \ref{def:causality}) that all its levels $ \lbrace t=t_0 \rbrace \coloneqq t^{-1}(t_0) $, $ t_0 \in \R $, are (necessarily spacelike) Cauchy hypersurfaces. This means that $ M $ is topologically a product $ \R \times D $, being $ t: \R \times D \rightarrow \R $ the natural projection on $ \R $ and $ \lbrace t=t_0 \rbrace = \{t_0\} \times D $ (see \cite{FS,Min}). The natural projection on $ D $ will be denoted by $ \pi: \R \times D \rightarrow D $.

In addition, in what follows $ S $ will denote a compact hypersurface with boundary of $ M $ included in $ \{t=0\} $. With abuse of notation, we will also denote by $ S $ its corresponding projection $ \pi(S) $ on $ D $. This way, $ \partial S $ is a co-dimension two spacelike submanifold of $ M $ and at each $ p \in \partial S $ there are exactly two lightlike directions $ L $-orthogonal to $ \partial S $ (see \cite[Prop. 5.2]{AJ}). For simplicity, we will assume that $ \partial/\partial t $ is timelike, so that one lightlike direction always points outwards from $ S $ and the other, inwards.\footnote{The timelike character of $ \partial/\partial t $ is useful to make this distinction between both lightlike directions, but it is not strictly necessary in most of the results (see the discussion in Rem. \ref{rem:partial_t} below).} Recall that $ L: A^* \rightarrow \R $ is any Lorentz-Finsler metric compatible with $ \mathcal{C} $, and consider its exponential map $ \text{exp} $.


\begin{lem}
\label{lem:H_hip}
Let $ N $ be a smooth causal vector field on $ \partial S $. There exists $ \varepsilon > 0 $ such that $ H \coloneqq \textup{exp}(\hat{H}_{\varepsilon}) $ is a (smooth) hypersurface of $ M $, where
\begin{equation*}
\hat{H}_{\varepsilon} \coloneqq \lbrace (p,\tau N_p) \in TM: p \in \partial S, \tau \in (-\varepsilon,\varepsilon) \rbrace.
\end{equation*}
\end{lem}
\begin{proof}
We will prove that there exists such an $ \varepsilon > 0 $ that makes $ \text{exp}|_{\hat{H}_{\varepsilon}} $ an embedding. First, choose $ \varepsilon_1 > 0 $ small enough for $ \hat{H}_{\varepsilon_1} $ to be included in the domain of definition of the exponential map.\footnote{This domain $ U \subset TM $ satisfies that $ U \setminus \textbf{0} $ is open and each $ U_p \coloneqq U \cap T_pM $ is star-shaped and contains points in all the causal directions.} We denote the zero section on $ \partial S $ as $ \partial\hat{S} \coloneqq \lbrace (p,0) \in TM / p \in \partial S \rbrace \subset TM $, which is a hypersurface of $ \hat{H}_{\varepsilon_1} $, which is in turn a submanifold of $ TM $. In general, $ \text{exp} $ is not smooth on the zero section, but here we will only work with its restriction to $ \hat{H}_{\varepsilon_1} $, which contains just one direction for each point, making $ \text{exp}|_{\hat{H}_{\varepsilon_1}} $ smooth. This map satisfies:
\begin{equation*}
\begin{array}{cccc}
\text{exp}|_{\hat{H}_{\varepsilon_1}} \colon & \hat{H}_{\varepsilon_1} \subset TM & \longrightarrow & M\\
& (p,v) & \longmapsto & \text{exp}_p(v)\\
& (p,0) & \longmapsto & p.
\end{array}
\end{equation*}
Also, its differential map on any $ (p,v) \in \hat{H}_{\varepsilon_1} $, whose domain of definition is $ T_{(p,v)}\hat{H}_{\varepsilon_1} \equiv T_p\partial S \oplus \text{Span}(N_p) $, satisfies:
\begin{equation*}
\begin{array}{cccc}
\df(\text{exp}|_{\hat{H}_{\varepsilon_1}})_{(p,v)} \colon & T_p\partial S \oplus \text{Span}(N_p) & \longrightarrow & T_{\text{exp}_p(v)}M\\
& (0,\tau N_p) & \longmapsto & \tau \df(\text{exp}_p)_{v}(N_p)\\
& (u,0) & \longmapsto & u.
\end{array}
\end{equation*}
Clearly $ \text{exp}|_{\partial\hat{S}} $ is an embedding, as $ \text{exp}|_{\partial\hat{S}} $ and $ \df(\text{exp}|_{\partial\hat{S}})_{(p,0)} $ are the identity maps. In particular, $ \df(\text{exp}|_{\partial\hat{S}})_{(p,0)} $ is injective at every $ (p,0) \in \partial\hat{S} $, so there must exist an open neighborhood $ \hat{V} \subset \hat{H}_{\varepsilon_1} $ of $ \partial\hat{S} $ such that, for all $ (p,v) \in \hat{V} $, $ \df({\text{exp}|_{\hat{V}}})_{(p,v)} $ is still injective. Reducing $ \varepsilon_1 $ if necessary, we can assume that $ \hat{V} = \hat{H}_{\varepsilon_2} $ for some $ \varepsilon_2 \leq \varepsilon_1 $ (thanks to the compactness of $ \partial S $) and therefore, $ \text{exp}|_{\hat{H}_{\varepsilon_2}} $ is an inmersion. As every injective inmersion of a compact manifold (with or without boundary) is an embedding, it suffices to show that there exists a compact submanifold of $ \hat{H}_{\varepsilon_2} $ where the exponential map is injective. We define, for each $ n \in \N $,
\begin{equation*}
U_n \coloneqq \left\lbrace (p,\tau N_p) \in TM: p \in \partial S, \tau \in \left[ \frac{-1}{n},\frac{1}{n} \right] \right\rbrace,
\end{equation*}
which are compact submanifolds (with boundary) of $ \hat{H}_{\infty} $ that verify $ \partial\hat{S} = \cap_{n \in \N} U_n $. Obviously there exists $ n_1 \in \N $ such that $ U_n \subset \hat{H}_{\varepsilon_2} $ for all $ n \geq n_1 $, so we only need to prove that the restriction of the exponential map to one of these $ U_n $ is injective. Suppose $ \text{exp}|_{U_n} $ is not injective for all $ n \in \N $. Then there exist two sequences $ \lbrace h_n \rbrace_{n \in \N} $ and $ \lbrace h'_n \rbrace_{n \in \N} $, with $ h_n, h'_n \in U_n $ and $ h_n \not= h'_n $ for all $ n \in \N $, such that $ \text{exp}(h_n) = \text{exp}(h'_n) $. Taking subsequences if necessary, we can ensure that $ h_n \rightarrow h \in \partial\hat{S} $ and $ h'_n \rightarrow h' \in \partial\hat{S} $. By continuity, $ \text{exp}(h) = \text{exp}(h') $ and since $ \text{exp}|_{\partial\hat{S}} $ is an embedding, we conclude that $ h = h' $. Now, we know that $ \text{exp}|_{\hat{H}_{\varepsilon_2}} $ is an inmersion, and every inmersion is locally an embedding, so there exists a neighborhood $ U_h $ of $ h \in \partial\hat{S} \subset \hat{V} $ where $ \text{exp}|_{U_h} $ is injective. But the sequences $ \lbrace h_n \rbrace_{n \in \N} $ and $ \lbrace h'_n \rbrace_{n \in \N} $ must enter $ U_h $, which contradicts the injectivity of $ \text{exp}|_{U_h} $. We conclude then that there exists $ n_2 \geq n_1 $ such that $ \text{exp}|_{U_{n_2}} $ is injective and thus an embedding. To end the proof, we can choose $ \varepsilon \coloneqq 1/n_2 $, so that $ \hat{H}_{\varepsilon} = \text{Int}(U_{n_2}) $; $ \text{exp}|_{\hat{H}_{\varepsilon}} $ will also be an embedding and consequently, $ H \coloneqq \text{exp}(\hat{H}_{\varepsilon}) $ is a (smooth) hypersurface of $ M $.
\qed
\end{proof}

\begin{lem}
\label{lem:cone_geod}
Each $ p \in \partial I^+(S) \setminus S $ lies on a cone geodesic entirely contained in $ \partial I^+(S) $ that departs from $ \partial S $.
\end{lem}
\begin{proof}
This is a consequence of \cite[Thm. 2.48]{Min}, which guarantees that each $ p \in \partial I^+(S) \setminus S $ lies on a ``lightlike geodesic'' (defined as a locally $\mathring{J}$-arelated causal curve; see \cite[Def. 2.6]{Min}) entirely contained in $ \partial I^+(S) $ and starting at $ S $. Note that, in our setting, this definition of lightlike geodesics coincides with that of cone geodesics (Def. \ref{def:cone_geod}), since for any causal curve $ \gamma $, $ \gamma(t_2) \notin \mathring{J}^+(\gamma(t_1)) = \text{Int}(J^+(\gamma(t_1))) $ if and only if $ \gamma(t_1) \rightarrow \gamma(t_2) $. Moreover, $ \gamma $ must depart from $ \partial S $ by continuity.
\qed
\end{proof}


We are now in a position to prove that $ \partial I^+(S) \setminus S $ is a (smooth) lightlike hypersurface for small $ t $. Although we provide a self-contained proof, some comments on its relationship with the results in \cite[\S 4.1]{JPS} are given in Rem. \ref{rem:rel_huygens}. 

\begin{thm}
\label{th:lig_hyp}
 Let $ (M,\mathcal{C}) $ be a globally hyperbolic  cone structure.  For any compact hypersurface with boundary $ S \subset \{t=0\} $, there exists $ \varepsilon > 0 $ such that $ \partial I^+(S) \cap ((0,\varepsilon) \times D) $ is a lightlike hypersurface. 
\end{thm}
\begin{proof}
Let $ N $ be the unique (smooth) lightlike vector field on $ \partial S $ such that $ N $ is $ L $-orthogonal to $ \partial S $ and points outwards, with $ \df t(N) = 1 $. Consider
\begin{equation*}
H \coloneqq \lbrace \text{exp}_p(\tau N_p): p \in \partial S, \tau \in (0,\varepsilon) \rbrace
\end{equation*}
for some $ \varepsilon > 0 $ small enough for Lem. \ref{lem:H_hip} to hold, so that $ H $ is a hypersurface in $ M $. We will see that $ R \coloneqq \partial I^+(S) \cap ((0,t_0) \times D) \subset H $ for some $ t_0 > 0 $.\footnote{In fact, we will prove later that $ R = H $ (see Prop. \ref{prop:cau_cur}).} Choose any $ q \in R $. By Lem. \ref{lem:cone_geod}, $ q $ lies on a cone geodesic entirely contained in $ \partial I^+(S) $ that starts at $ p \in \partial S $. We can assume that the initial velocity of this cone geodesic is given by a lightlike vector $ \tilde{N}_p $ normalized to $ \df t(\tilde{N}_p) = 1 $. If $ \tilde{N}_p $ is not $ L $-orthogonal to $ \partial S $ at $ p $, there exists a timelike curve from $ \partial S $ to $ q $ (see \cite[Prop. 6.4]{AJ}), so $ q \in I^+(\partial S) \subset I^+(S) $ and then $ q \notin \partial I^+(S) $ ($ I^+(S) $ is open), which is a contradiction. Thus we have that $ \tilde{N}_p $ is $ L $-orthogonal to $ \partial S $ at $ p $, and it also has to point out to the exterior of $ S $ because otherwise, the cone geodesic would enter $ I^+(S) $, in contradiction with the fact that it is contained in $ \partial I^+(S) $. We conclude then that $ \tilde{N}_p = N_p $. The unicity of geodesics with the same initial conditions guarantees that, having chosen a sufficiently small $ t_0 > 0 $, $ q = \text{exp}_p(\tau_0 N_p) $ for some $ \tau_0 \in (0,\varepsilon) $ (note that, in this case where $ \df t(N) = 1 $, the cone geodesics are parametrized with respect to $ t $, i.e., $ t \circ \text{exp}(\tau N) = \tau $, so we can choose $ t_0 = \varepsilon $), thus $ q \in H $. This proves that $ R \subset H \subset M $, being $ R $ already a topological hypersurface included in a (smooth) hypersurface $ H $. Therefore, the topological hypersurface $ R $ must also be smooth.

We prove now that $ R $ is lightlike. For any $ p \in R $, there are three possible cases: 
\begin{itemize}
\item $ T_pR \cap \mathcal{C}_p = \emptyset $ $ \Rightarrow $ $ T_pR $ is spacelike. This is a contradiction, since we know that there is at least one lightlike direction $ v $ in $ T_pR $ (as $ p $ lies on a cone geodesic).

\item $ T_pR $ is transverse to $ \mathcal{C}_p $ $ \Rightarrow $ there are two independent lightlike directions in $ T_pR $. But this leads again to a contradiction, because if there was another lightlike direction $ u \in T_pR $ independent of $ v $, then $ v+u $ would be a timelike direction: $ v+u \in \overline{A}_p $ as $ A_p $ (the $ \mathcal{C}_p $-interior) is convex, but $ v+u \notin \mathcal{C}_p $ because of the (non-radial) strong convexity of $ \mathcal{C}_p $; so there would be a timelike curve in $ R $, which is achronal. Therefore, there is a unique lightlike direction $ v $ at each $ T_pR $, but not a timelike one.

\item $ T_pR $ is tangent to $ \mathcal{C}_p $, necessarily along the unique lightlike direction $ v $ $ \Rightarrow $ $ T_pR = T_v\mathcal{C}_p $.
\end{itemize}

Therefore, we conclude that $ R $ is a lightlike hypersurface (recall Def. \ref{def:lig_hyp}).
\qed
\end{proof}

\begin{rem}
The previous theorem also holds if we substitute $ \partial I^+(S) \setminus S $ for $ \partial I^+(\partial S) \setminus \partial S $ (and therefore consider also the inward lightlike direction) but, in this case, we would obtain two disjoint lightlike hypersurfaces.
\end{rem}

\begin{cor}
\label{cor:N}
Let $ R := \partial I^+(S) \cap ((0,\varepsilon)\times D) $, with $ \varepsilon $ small enough for Thm. \ref{th:lig_hyp} to hold, and let $ N $ be its associated lightlike vector field. Then $ N $ is the unique (up to a positive pointwise scale factor) lightlike vector field pointing outwards that is $ L $-orthogonal to $ \partial S $ and any $ R \cap \lbrace t = \tau \rbrace $, $ 0 < \tau < \varepsilon $.
\end{cor}
\begin{proof}
For any $ \tau \in (0,\varepsilon) $, $ R \cap \lbrace t = \tau \rbrace $ is a co-dimension two spacelike submanifold of $ M $, along which $ N $ is $ L $-orthogonal (recall Cor. \ref{prop:tra_int}). By continuity, $ N $ is also $ L $-orthogonal to $ \partial S $. Moreover, as with $ \partial S $, at each $ p \in R \cap \lbrace t = \tau \rbrace $ there are exactly two lightlike directions $ L $-orthogonal to $ R \cap \lbrace t = \tau \rbrace $, one pointing outwards and the other, inwards (see \cite[Prop. 5.2]{AJ}). Specifically, $ N $ has to point outwards because it is tangent to $ R $.
\qed
\end{proof}

\section{Minimization properties of cone geodesics}
\label{sec:min_geod}
One of the main properties of lightlike geodesics in globally hyperbolic Lorentzian spacetimes is that they locally minimize the arrival time (the well-known Fermat's principle). The framework and results we have developed so far enable us to directly extend  this result to the cone structures setting, and even provide a  generalized version. Specifically, we will show that cone geodesics departing $ L $-orthogonally from $ \partial S $ (and pointing outwards) minimize the propagation time from $ S $ (i.e., the arrival time to any observer given by an integral curve of $ \partial/\partial t $).

We still follow in this section the conventions and notation established at the beginning of \S \ref{sec:smooth}. In addition, recall the notation $ R \coloneqq \partial I^+(S) \cap ((0,\varepsilon)\times D) $, with $ \varepsilon $ small enough for Thm. \ref{th:lig_hyp} to hold, and $ N $ will denote its associated  lightlike  vector field.

\begin{prop}
\label{prop:cau_cur}
The only causal curves from $ S $ entirely contained in $ R $ are all the cone geodesics starting at $ \partial S $ with initial velocity $ L $-orthogonal to $ \partial S $ and pointing outwards. In fact, these cone geodesics are integral curves of $ N $ when suitably parametrized.
\end{prop}
\begin{proof}
If $ \gamma $ is a causal curve contained in $ R $, then its velocity $ \gamma' $ must be lightlike (recall Prop. \ref{prop:lig_hyp}) and thus proportional to $ N $ at each point. Without loss of generality, we can assume that $ N $ is normalized and $ \gamma $ parametrized in such a way that $ \gamma'(t) = N_{\gamma(t)} $ (e.g., if $ \df t(N)=1 $ then $ \gamma $ has to be parametrized with respect to $ t $), i.e., $ \gamma $ is an integral curve of $ N $ and hence a cone geodesic (by Prop. \ref{prop:int_cur_geod}) with initial velocity $ L $-orthogonal to $ \partial S $ and pointing outwards (recall Cor. \ref{cor:N}). Conversely, if $ \gamma $ is a cone geodesic with initial velocity $ \gamma'(0) $ $ L $-orthogonal to $ \partial S $ and pointing outwards, then $ \gamma'(0) = N_{\gamma(0)} $ up to reparametrizations (again by Cor. \ref{cor:N}). By the unicity of geodesics with the same initial conditions, $ \gamma $ must coincide with the integral curve of $ N $ starting at $ \gamma(0) $, which is contained in $ R $.
\qed
\end{proof}

\begin{thm}
\label{th:min_time}
 Within the hypothesis of Thm. \ref{th:lig_hyp} (with timelike $ \partial/\partial t $), let $ \gamma $ be a cone geodesic departing $ L $-orthogonally from $ \partial S $ and pointing outwards. Then for any $ p_0 = (t_0,x_0) \in \textup{Im}(\gamma) \subset \R \times D $,  with  $ t_0 < \varepsilon $, $ \gamma $ is the causal curve from $ S $ that arrives strictly first at the vertical line $ l_{x_0}: t \mapsto (t,x_0) $ (up to reparametrizations).
\end{thm}
\begin{proof}
Fix any $ p_0 = (t_0,x_0) \in \text{Im}(\gamma) $ and suppose there is a causal curve $ \varphi $, different from $ \gamma $, that goes from $ S $ to $ (t_1,x_0) $, with $ t_1 \leq t_0 $. If $ t_1 = t_0 $, $ \varphi $ has to be contained in $ \partial I^+(S) $ because otherwise, it would enter $ I^+(S) $ and could not reach $ (t_0,x_0) \in \partial I^+(S) $ (this is a consequence of \cite[Prop. 6.5]{AJ}). But then $ \gamma $ and $ \varphi $ are both integral curves of $ N $ (by Prop. \ref{prop:cau_cur}) arriving at the same point, so $ \gamma = \varphi $ (up to reparametrizations), which contradicts the initial assumption. If $ t_1 < t_0 $, we can construct the (piecewise smooth) curve given by $ \varphi $ from $ S $ to $ (t_1,x_0) $, and by the timelike vertical line $ l_{x_0} $ from $ (t_1,x_0) $ to $ (t_0,x_0) $. Then, by \cite[Prop. 6.5]{AJ} there exists a timelike curve from $ S $ to $ (t_0,x_0) $ and therefore $ (t_0,x_0) \in I^+(S) $, which is a contradiction.
\qed
\end{proof}

\begin{rem}
\label{rem:partial_t}
The assumption made at the beginning of \S \ref{sec:smooth} that $ \partial/\partial t $ is timelike only becomes truly crucial in the previous theorem. Indeed, $ \partial/\partial t $ plays the role of an observers' vector field, and the cone geodesics of the previous result are time-minimizing with respect to the time these observers measure. In Relativity, no observer can move faster than light, which ensures that $ \partial/\partial t $ is timelike. However, cone structures can also be applied in non-relativistic settings to describe the propagation of a wave that propagates through a medium \cite{CJS,JPS}. When the medium moves with respect to $ \partial/\partial t $ faster than the wave, $ \partial/\partial t $ becomes spacelike and the cone geodesics that go against the current lose the property of being time-minimizing with respect to the time measured by $ \partial/\partial t $. Nevertheless, in this case we can define an observers' vector field co-moving with the medium (and therefore, timelike), providing a new decomposition of $ M $ as a product $ \R \times D $ with respect to which the previous theorem would still hold (see \cite[\S 6]{JPS}).

Note that using this technique, even if the original $ \partial/\partial t $ was not timelike, we could always select a new temporal fuction $ \tilde t $ with $ \partial/\partial\tilde t $ timelike in a different decomposition of $ M $.  Therefore, every result we have stated before Thm. \ref{th:min_time} is still valid when $ \partial/\partial t $ is not timelike, since being a lightlike hypersurface (the key result from which the others follow) is independent of the choice of the temporal function. The previous theorem, however, is an exception because the time-minimizing property is measured with respect to the selected temporal function.
\end{rem}

Anyway, whether $ \partial/\partial t $ is timelike or not, we can still ensure the existence of time-minimizing (among causal curves) cone geodesics  (see also \cite[Thm.~2.49]{Min}).\footnote{This guarantees the existence of solution to Zermelo's navigation problem, which seeks the fastest trajectory between two prescribed points for a moving object with respect to a medium; see \cite{CJS}.}

\begin{prop}
\label{cor:uni_con_geod}
For any $ x_0 \in \pi(J^+(S)) \setminus S \subset D $ there exists a time-minimizing cone geodesic from $ S $ to $ l_{x_0} $, necessarily contained in $ \partial I^+(S) $.
\end{prop}
\begin{proof}
Since $ J^+(S) $ is closed (due to the  global  hyperbolicity of $ M $; see \cite[\S 2.10]{Min}), there exists $ t_0 = \min\{ t \in \R: (t,x_0) \in J^+(S) \} $. Then $ (t_0,x_0) \in \partial I^+(S) $ and by Lem. \ref{lem:cone_geod} there exists a cone geodesic from $ \partial S $ to $ (t_0,x_0) $ entirely contained in $ \partial I^+(S) $. Clearly this cone geodesic must be time-minimizing among causal curves from $ S $, since any $ (t,x_0) $ with $ t < t_0 $ does not belong to $ J^+(S) $ and therefore, it is not reachable by any causal curve from $ S $.
\qed
\end{proof}

\begin{rem}
\label{rem:rel_huygens}
 Observe that all the results in this section have been obtained by exploiting the smoothness of $ \partial I^+(S) \setminus S $, in the spirit of making this work as self-contained as possible. A different approach focusing on the null cut function is carried out in \cite[\S 4.1]{JPS} with an analogous outcome. Specifically, we stress that the following statements are equivalent (in globally hyperbolic cone structures with timelike $ \partial/\partial t $): 
\begin{itemize}
\item \ $ \partial I^+(S) \setminus S $ is smooth (and thus a lightlike hypersurface) in $ (0,\varepsilon)\times D$, for some $ \varepsilon > 0 $ (Thm.~\ref{th:lig_hyp}). 
\item \ Any cone geodesic departing $ L $-orthogonally from $ \partial I^+(S) $ (and pointing outwards) does  not reach its cut point (i.e., does  not leave $ \partial I^+(S) $) in the region $(0,\varepsilon)\times D $ (\cite[Thm. 4.8]{JPS}). 
\item  Any causal curve $ \gamma $ entirely contained in $ \partial I^+(S) $ is necessarily a cone geodesic $ L $-orthogonal to $ \partial S $ (Prop. \ref{prop:cau_cur} and \cite[Prop. 4.4]{JPS}) that minimizes the propagation time from $ S $ while in the region $(0,\varepsilon)\times D$  (Thm. \ref{th:min_time} and \cite[Lem. 4.7]{JPS}). 
\end{itemize}

Also, it is worth mentioning that \cite{JPS} studies the general case in which $ S $ is a submanifold of arbitrary co-dimension.
\end{rem}

\section*{Acknowledgments}
The authors warmly acknowledge Prof. M. S\'{a}nchez (Universidad de Granada) for the careful reading of a preliminary version of this work and for his ever helpful advice and comments. This work is a result of the activity developed within the framework of the Programme in Support of Excellence Groups of the Regi\'on de Murcia, Spain, by Fundaci\'on S\'eneca, Science and Technology Agency of the Regi\'on de Murcia. MAJ was partially supported by MICINN/FEDER project reference PGC2018-097046-B-I00 and Fundaci\'on S\'eneca (Regi\'on de Murcia) project reference 19901/GERM/15, Spain, and EPR by MINECO/FEDER project reference MTM2016-78807-C2-1-P, MICINN/FEDER project reference PGC2018-097046-B-I00 and Contratos Predoctorales FPU-Universidad de Murcia, Spain.

%
%

\end{document}